\documentclass{amsart}
\usepackage{amssymb}
\usepackage{stmaryrd}
\usepackage{amsfonts}
\usepackage{amsmath}
\usepackage{url}
\usepackage{color}

\newtheorem{theorem}{Theorem}

\newtheorem{proposition}[theorem]{Proposition}
\newtheorem{corollary}[theorem]{Corollary}
\newtheorem{lemma}[theorem]{Lemma}
\newtheorem{claim}[theorem]{Claim}
\newtheorem{definition}[theorem]{Definition}
\newtheorem{assumption}[theorem]{Assumption}
\newtheorem{observation}[theorem]{Observation}

\theoremstyle{remark}
\newtheorem{remark}[theorem]{Remark}

\newcommand{\nequiv}{\equiv \hspace{-9pt} \slash \hspace{4pt}}

\title{Delaunay triangulations of lens spaces}
\date{January 2009.
}

\author{Fran\c{c}ois Gu\'eritaud}

\begin{document}

\begin{abstract}
We compute the convex hull $\Pi$ of an arbitrary finite subgroup $\Gamma$ of ${\mathbb{C}^*}^2$ --- or equivalently, of a generic orbit of the action of $\Gamma$ on $\mathbb{C}^2$. The basic case is $\Gamma=\{(e^{2ik\pi/q},e^{2ikp\pi/q})~|~0\leq k<q\}$ where $p\in\llbracket 2,q-2\rrbracket$ is coprime to $q$: then, $\Pi$ projects to a canonical or ``Delaunay'' triangulation $\mathcal{D}$ of the lens space $L_{p/q}=\mathbb{S}^3/\Gamma$ (endowed with its spherical metric), and the combinatorics of $\mathcal{D}$ are dictated by the continued fraction expansion of $p/q$.
\end{abstract}

\maketitle
\section{Introduction}

Given a compact pointed Riemannian $3$--manifold $(M,x_0)$, a natural object to construct is the Voronoi domain of $x_0$, i.e.~the set $X$ of all points $x$ such that the shortest path from $x$ to $x_0$ is unique. This domain $X$ can be embedded as a contractible subset of the universal cover $\widetilde{M}$ of $M$; if $M$ is homogeneous, then $X$ is typically (though not always) the interior of a polyhedron whose faces are glued in pairs to yield $M$. If so, dual to $X$ (and this gluing data) is the so-called Delaunay decomposition $\mathcal{D}$ of $M$, which comprises one cell per vertex of $X$, and has only one vertex, namely $x_0$. If $\widetilde{M}$ is $\mathbb{S}^3$ or $\mathbb{R}^3$ or $\mathbb{H}^3$, it is a classical result that $\mathcal{D}$ is itself realized by geodesic polyhedra which tile $M$.

A strong motivation for studying the Delaunay decomposition is that it is a combinatorial invariant of $(M,x_0)$ that encodes all the topology of $M$; this also suggests that computing $\mathcal{D}$ is hard in general. Jeff Weeks' program SnapPea \cite{snappea} achieves this numerically in the cusped hyperbolic case (taking $x_0$ in the cusp); for explicit theoretical predictions of $\mathcal{D}$ in special cases, see for example \cite{these, aswy, 
lackenby, qf, ananas}.

This paper is primarily concerned (Sections \ref{sec:prelim} through \ref{sec:proof}) with the case $M=\mathbb{S}^3/\varphi$, where 
$$\varphi(z,z')=\left (e^{\frac{2i\pi}{q}}z,e^{\frac{2ip\pi}{q}}z'\right )$$ 
and $\mathbb{S}^3$ is seen as the unit sphere of $\mathbb{C}^2$. Here, $\frac{p}{q}$ is a rational of $(0,1)$ in reduced form, and $M$ is called the \emph{lens space} $L_{p/q}$. We will show that the combinatorics of $\mathcal{D}$ (and $X$) are dictated by the continued fraction expansion of $\frac{p}{q}$ (and are independent of the choice of basepoint $x_0$).
 
The lift of $\mathcal{D}$ to $\mathbb{S}^3$ is the Delaunay decomposition of $\mathbb{S}^3$ with respect to a \emph{finite set} $\langle \varphi \rangle \widetilde{x}_0$ of vertices. Finally, in Section \ref{sec:general}, we extend our results to the case where $\langle \varphi \rangle$ is replaced by an arbitrary finite subgroup of $\mathbb{S}^1\times \mathbb{S}^1$ (possibly non-cyclic, acting possibly with fixed points on $\mathbb{S}^3$). 

\subsection*{History} After the first version of this paper was posted, G\"unter M. Ziegler made me aware of Smilansky's paper \cite{smilansky2} where essentially the same results were proven. The approaches are similar, except for the key result: we prove the convexity of a certain plane curve $\gamma$ by a big computation (Claim \ref{cla:key}); Smilansky in \cite{smilansky2} seems unaware that $\gamma$ is always convex, but has a clever lemma (proved in \cite{smilansky1}) to show that $\gamma$ behaves ``as though it were convex'' with respect to certain intersecting lines.

Note that Sergei Anisov has also announced similar results in \cite{anisov1, anisov2}.

\subsection*{Acknowledgements}
The main result (without its proof!) occurred to me during the workshop on Heegaard splittings at AIM, Palo Alto, in December 2007. It is a pleasure to thank the organizers of this beautiful meeting, as well as Omprakash Gnawali for early computer experiments and Saul Schleimer for subsequent discussions on the topic.

\section{Preliminaries} \label{sec:prelim}

Let $x_0$ be a point of $\mathbb{S}^3$ and $\mathcal{O}\subset \mathbb{S}^3$ its $\langle \varphi \rangle$--orbit. Suppose that the convex hull $\Pi$ of $\mathcal{O}$ has non-empty interior. It is well-known that the boundary of $\Pi$ then decomposes into affine cells, whose projections to $\mathbb{S}^3$ (from the origin) are precisely the cells of the Delaunay decomposition $\mathcal{D}$. Therefore, all we have to do is to determine the faces of the convex hull $\Pi$ of $\mathcal{O}$: these are Theorems \ref{thm:main} and \ref{thm:nofaces} below.

\subsection{What is the generic case?}

However, if $p\equiv\pm 1~[\text{mod } q]$, then any orbit $\mathcal{O}$ of $\varphi$ is a regular polygon contained in a plane of $\mathbb{R}^4\simeq\mathbb{C}^2$, which easily implies that the Voronoi domain $X$ of $L_{p/q}$ (for any basepoint) is bounded by only two spherical caps (this is a special case where $X$ is \emph{not} a proper spherical polyhedron). It is also easy to see that the isometry group of $L_{p/q}$ acts transitively on $L_{p/q}$ in that case.

Therefore, we will assume $p\notin\{1,q-1\}$. Then, the identity component of the isometry group of $L_{p/q}$ lifts to the group $G=\mathbb{S}^1\times \mathbb{S}^1$ acting diagonally on $\mathbb{C}^2$ (of course, $\varphi \in G$). The $G$-orbits in $\mathbb{S}^3$ are the tori $\{(z,z')~|~\frac{|z'|}{|z|}=\kappa \}$ for $\kappa \in \mathbb{R}_+^*$, and the circles $C=\{0\}\times \mathbb{S}^1$ and $C'=\mathbb{S}^1\times \{0\}$. If $x_0\in C\cup C'$, then the orbit $\mathcal{O}=\langle \varphi \rangle x_0$ is a plane regular polygon, so the Voronoi domain $X$ is again bounded by two spherical caps. 

Therefore, we will be concerned with the generic case $x_0\in\mathbb{S}^3\smallsetminus (C\cup C')$. Since changing $x_0$ only modifies its orbit $\mathcal{O}$ (and therefore the polyhedron $\Pi$) by a diagonal automorphism of $\mathbb{C}^2$, all basepoints $x_0 \notin C\cup C'$ are equivalent as regards the combinatorics of $\Pi$ and of the Delaunay decomposition. In fact $x_0$ does not even need to belong to the \emph{unit} sphere: for convenience, we will take $x_0=(1,0,1,0)\in\sqrt{2}\mathbb{S}^3$ in Theorem \ref{thm:main}.

\subsection{An intuitive description of the triangulation}

Clearly, $L_{p/q}$ is obtained by gluing two solid tori $\{(z,z')\in \mathbb{S}^3~|~\frac{|z|}{|z'|}\geq 1\}/\varphi$ and $\{(z,z')\in \mathbb{S}^3~|~\frac{|z|}{|z'|}\leq 1\}/\varphi$, boundary-to-boundary. Equivalently, $L_{p/q}$ is a thickened torus $(\mathbb{S}^1)^2\times [0,1]$, attached to two thickened disks (one for each boundary component, along possibly very different slopes $s,s'$) and capped off with two balls. 

We now sketch a way of triangulating $L_{p/q}$ that emulates this construction: although it will not be needed in the sequel, it might provide some geometric intuition (the triangulation described here will turn out to be combinatorially equivalent to the Delaunay decomposition of $L_{p/q}$).

Consider the standard unit torus $T:=\mathbb{R}^2/\mathbb{Z}^2$ decomposed into two simplicial triangles, $(0,0)(0,1)(1,1)$ and $(0,0)(1,0)(1,1)$. We can simplicially attach two faces of a tetrahedron $\Delta$ to $T$, so that $\Delta$ materializes an \emph{exchange of diagonals} in the unit square. The union $T\cup\Delta$ is now a (partially) thickened torus, whose top and bottom boundaries are triangulated in two different ways. We can attach a new tetrahedron $\Delta'$, e.g. to the top boundary, so as to perform a new exchange of diagonals. Iterating the process many times, we can obtain a triangulation of (possibly a retract of) $T\times [0,1]$ with top and bottom triangulated (into two triangles each) in two essentially arbitrary ways. Finally, there exists a standard way of folding up the top boundary $T\times \{1\}$ on itself, identifying its two triangles across an edge: this was perhaps first formulated that way in \cite{jaco}. The result after folding-up is a \emph{solid torus}, also described with many pictures in \cite{ananas}. (In that paper, we show that such triangulated solid tori also arise naturally in the Delaunay decompositions of many \emph{hyperbolic} manifolds, namely, large ``generic'' Dehn fillings.) If we fold up the bottom $T\times \{0\}$ in a similar way, it turns out we can get any $L_{p/q}$ with $p\nequiv \pm 1 ~[\text{mod }q]$.

The main theorems below (\ref{thm:main} and \ref{thm:nofaces}) describe this same triangulation in a way that is self-contained and completely explicit, although perhaps less synthetic or helpful than the process described above. The interested reader may infer the equivalence of the two descriptions from the proof of Theorem \ref{thm:nofaces}; see also \cite{ananas}.

\subsection{Strategy}

Let $\mathbb{T}:=(\mathbb{R}/2\pi\mathbb{Z})^2$ be the standard torus and $\iota:\mathbb{T}\rightarrow \mathbb{C}^2\simeq\mathbb{R}^4$ denote the standard injection, satisfying $$\iota(u,v)=(\cos u, \sin u, \cos v, \sin v).$$ The subgroup $\Gamma:=\{\tau_k=(k\frac{2\pi}{q}, kp\frac{2\pi}{q})\}_{k\in \mathbb{Z}}$ of $\mathbb{T}$ is such that $\iota(\Gamma)=\mathcal{O}$, the orbit of $(1,0,1,0)\in \mathbb{R}^4$ under $\varphi$. Therefore, each top-dimensional cell (tetrahedron, as it turns out) in $\partial \Pi$ is spanned by the images under $\iota$ of four points $\tau,\tau',\tau'',\tau'''$ of $\Gamma$.

Our main result, Theorem \ref{thm:main}, claims that $\tau,\dots,\tau'''$ are the vertices of certain \emph{parallelograms} of $\mathbb{T}$ with the minimal possible area, namely $\frac{(2\pi)^2}{q}$. To prove this, the strategy is to consider a linear form $\rho:\mathbb{R}^4\rightarrow \mathbb{R}$ that takes the same value, say $Z>0$, on $\iota(\tau),\dots,\iota(\tau''')$; then look (e.g. in the chart $[-\pi,\pi]^2$) at the level curve $\gamma=(\rho\circ\iota)^{-1}(Z)$.

Lemma \ref{lem:convex} says that if $Z$ and the coefficients of $\rho$ satisfy certain inequalities, then $\gamma$ is a \emph{convex} Jordan curve passing through $\tau,\dots,\tau'''$. Intuitively, if the hyperplane $\rho^{-1}(Z)$ passes \emph{far enough} from the origin of $\mathbb{R}^4$ (in a sense depending on the direction of $\ker \rho$), it will only skim a small cap off $\iota(\mathbb{T})$ that looks convex in the chart. Convexity is key: it will imply that no other point of $\Gamma$ than $\tau,\dots,\tau'''$ lies inside $\gamma$, i.e.~in $(\rho\circ\iota)^{-1}[Z,+\infty)$. In other words, $\rho^{-1}(Z)\supset \iota(\{\tau,\dots,\tau'''\})$ is a supporting plane of the convex hull of $\iota(\Gamma)=\mathcal{O}$.

Proving that $Z$ and the coefficients of $\rho$ satisfy the inequalities of Lemma \ref{lem:convex} will be the trickier part of the work, done in Section \ref{sec:proof} using only basic trigonometry.

\subsection{Notation}

Until the end of Section \ref{sec:proof}, we fix $q\geq 5$ and $p\in \llbracket 2,q-2 \rrbracket$ coprime to $q$, so that $Q:=\frac{p}{q}$ is a rational of $(0,1)$ in reduced form. We denote by $x_0$ the point $(1,1)$ of $\mathbb{C}^2$, and by $x_k$ the $k$-th iterate of $x_0$ under the map $\varphi: (z,z') \mapsto ( e^{\frac{2i\pi}{q}} z, e^{\frac{2ip\pi}{q}} z')$. Finally we let $\Pi$ be the convex hull of $x_0,\dots,x_{q-1}$. We identify $\mathbb{R}^4$ with $\mathbb{C}^2$ in the standard way. The transpose of a matrix $M$ is written $M^t$.

By \emph{Farey graph}, we mean the graph obtained by connecting two rationals $\frac{\alpha}{a}, \frac{\beta}{b}$ of $\mathbb{P}^1\mathbb{R}=\partial_{\infty}\mathbb{H}^2$ by a geodesic line in $\mathbb{H}^2$ whenever $|\alpha b -\beta a|=1$ (this graph consists of the ideal triangle $01\infty$ reflected in its sides \emph{ad infinitum}, and $\text{PSL}_2\mathbb{\mathbb{Z}}\subset \text{PSL}_2\mathbb{R}\simeq \text{Isom}^+(\mathbb{H}^2)$ acts faithfully transitively on oriented edges). For example, two rationals connected by a Farey edge are called \emph{Farey neighbors}. Refer to \cite{farey} for the classical 
casting of continued fractions in terms of the Farey graph.

\section{Main result: description of the faces of $\Pi$} \label{sec:mainresult}

\begin{theorem}
Let $A=\frac{\alpha}{a}, B=\frac{\beta}{b} \in [0,1]$ be Farey neighbors such that $Q=\frac{p}{q}$ lies strictly between $A$ and $B$, at most one of $A,B$ is a Farey neighbor of $Q$, and at most one of $A,B$ is a Farey neighbor of $\infty$ (i.e.~belongs to $\{0,1\}$). Then $x_0,x_a,x_b,x_{a+b}$ span a top-dimensional cell (tetrahedron) of $\Pi$. \label{thm:main}
\end{theorem}
Note that in the simplest case $\frac{p}{q}=\frac{2}{5}$, there is only one pair $\{\frac{\alpha}{a},\frac{\beta}{b}\}=\{\frac{1}{3},\frac{1}{2}\}$.
Theorem \ref{thm:main} will be proved in Section \ref{sec:proof}. Meanwhile, we check (Theorem \ref{thm:nofaces}) that there are no \emph{other} top-dimensional faces in $\partial \Pi$. Note that we make no assumption on whether $A<B$ or $B<A$, or on whether $a<b$ or $b<a$ (all four possibilities can arise), so we will always be able to switch $A$ and $B$ for convenience. 

\begin{remark} \label{rem:order}
It is well-known that the number of unordered pairs of rationals $\{\frac{\alpha}{a},\frac{\beta}{b}\}$ satisfying the hypotheses of Theorem \ref{thm:main} is $n-3$, where $n$ is the sum of all coefficients of the continued fraction expansion 
of $Q$. Moreover, these pairs are naturally ordered: the first pair is $\{\frac{0}{1},\frac{1}{2}\}$ or $\{\frac{1}{2},\frac{1}{1}\}$ according to the sign of $Q-\frac{1}{2}$; the pair coming after $\{\frac{\alpha}{a},\frac{\beta}{b}\}$ is either $\{\frac{\alpha}{a},\frac{\alpha+\beta}{a+b}\}$ or $\{\frac{\alpha+\beta}{a+b},\frac{\beta}{b}\}$. Reversing this, the pair coming \emph{before} $\{\frac{\alpha}{a},\frac{\beta}{b}\}$ is $\{\frac{\min (\alpha,\beta)}{\min(a,b)},\frac{|\alpha-\beta|}{|a-b|}\}$. The last pair $\{\frac{\alpha}{a},\frac{\beta}{b}\}$ contains exactly one Farey neighbor of $\frac{p}{q}$ and is such that $\frac{\alpha+\beta}{a+b}$ is another Farey neighbor of $\frac{p}{q}$: therefore that last pair satisfies either $\frac{\alpha+(\alpha+\beta)}{a+(a+b)}=\frac{p}{q}$ or $\frac{(\alpha+\beta)+\beta}{(a+b)+b}=\frac{p}{q}$. \end{remark}

\begin{theorem} \label{thm:nofaces}
All top-dimensional faces of $\Pi$ are tetrahedra whose vertices are of the form $x_n x_{n+a} x_{n+b} x_{n+a+b}$ with $a,b$ as in Theorem \ref{thm:main}, and $n\in\mathbb{Z}$.
\end{theorem}
\begin{proof}
Assuming Theorem \ref{thm:main}, it is enough to find a tetrahedron of the given form, adjacent to every face of the tetrahedron $T_{a,b}:=x_0 x_a x_b x_{a+b}$ (but possibly with a different pair $\{a,b\}$).

\medskip

First, the faces of $T_{a,b}$ obtained by dropping $x_0$ or $x_{a+b}$ indeed have neighbors:

If $T_{a,b}$ is the first tetrahedron for the ordering, Remark \ref{rem:order} implies $T_{a,b}=T_{1,2}=x_0x_1x_2x_3$. The face $x_0x_1x_2$ of $T_{a,b}$ (obtained by dropping $x_3$) is adjacent to $\varphi^{-1}(T_{a,b})=x_{-1}x_0x_1x_2$, and similarly the face $x_1x_2x_3$ obtained by dropping $x_0$ is adjacent to $\varphi(T_{a,b})=x_1x_2x_3x_4$.

If $T_{a,b}$ is \emph{not} the first tetrahedron, then we can assume $a<b$ and by Remark \ref{rem:order} there is a previous tetrahedron $T_{b-a,a}$. The face $x_0 x_a x_b$ of $T_{a,b}$ is adjacent to $T_{b-a,a}=x_0 x_{b-a} x_a x_b$; the face $x_a x_b x_{a+b}$ of $T_{a,b}$ is adjacent to $\varphi^a(T_{b-a,a})=x_a x_b x_{2a} x_{a+b}$.

\medskip

Lastly, the faces of $T_{a,b}$ obtained by dropping $x_a$ or $x_b$ also have neighbors:

If $T_{a,b}$ is the last tetrahedron, then Remark \ref{rem:order} implies $a+2b=q$ (up to switching $a,b$), hence $x_{a+b}=x_{-b}$ (because $x_q=x_0$). Therefore the face $x_0 x_a x_{a+b}=x_0 x_a x_{-b}$ of $T_{a,b}$ is adjacent to $\varphi^{-b}(T_{a,b})=x_{-b}x_{a-b}x_0 x_a$, and the face $x_0 x_b x_{a+b}=x_{a+2b} x_b x_{a+b}$ of $T_{a,b}$ is adjacent to $\varphi^{b}(T_{a,b})=x_{b}x_{a+b}x_{2b} x_{a+2b}$.

If $T_{a,b}$ is \emph{not} the last tetrahedron, then up to switching $a,b$ there is, by Remark 2, a next tetrahedron $T_{a,a+b}$. Therefore the face $x_0 x_a x_{a+b}$ of $T_{a,b}$ is adjacent to $T_{a,a+b}=x_0 x_a x_{a+b} x_{2a+b}$, and the face $x_0 x_b x_{a+b}$ of $T_{a,b}$ is adjacent to $\varphi^{-a}(T_{a,a+b})=x_{-a} x_0 x_b x_{a+b}$.
\end{proof}

\section{Main tools} \label{sec:tools}

Under the assumptions of Theorem \ref{thm:main}, and before we start its proof proper, let us introduce some tools. These are of two types: arithmetic properties of the integers appearing in the Farey diagram (Section \ref{sec:farey}), and geometric properties of the standard embedding $\iota$ of $\mathbb{S}^1\times \mathbb{S}^1$ into $\mathbb{R}^2\times \mathbb{R}^2$ (especially its intersections with hyperplanes), in Section \ref{sec:tori}.

\subsection{Farey relationships on integers} \label{sec:farey}

Let $X=\frac{\xi}{x}=\frac{\alpha+\beta}{a+b}$ and $Y=\frac{\eta}{y}=\frac{|\alpha-\beta|}{|a-b|}$ be the two common Farey neighbors of $A$ and $B$ ($X$ is closer to $Q$ while $Y$ is closer to $\infty=\frac{1}{0}$; we have $X,Y\in[0,1]$). We introduce the notation
$$\frac{u}{v} \wedge \frac{s}{t}:=|ut-vs|$$
for any two rationals $\frac{u}{v}, \frac{s}{t}$ in reduced form. For example, if $h,h'$ are rational, then $h\wedge h'=1$ if and only if $h,h'$ are Farey neighbors; moreover, the denominator of $h$ is always equal to $h\wedge \infty$.

\medskip
\noindent We thus have $\left \{\begin{array}{c}
a=A\wedge\infty \\
b=B\wedge\infty \\
x=X\wedge\infty \\
y=Y\wedge\infty \\
q=Q\wedge\infty  \end{array} \right .$
and we define $\left \{\begin{array}{c}
a':=A\wedge Q \\
b':=B\wedge Q \\
x':=X\wedge Q \\
y':=Y\wedge Q \end{array} \right .$, all positive.

\begin{proposition} \label{prop:abab}
One has $\left \{\begin{array}{ccc} a+b&=&x \\ |a-b|&=&y \end{array} \right .$, and $\left \{ \begin{array}{ccc}a'+b'&=&y' \\ |a'-b'|&=& x' \end{array}\right .$, and $$a'b+b'a=q.$$
\end{proposition}
\begin{proof}
The first two identities are obvious from the definitions of $X,Y$. For the next two identities, notice that $\alpha q-ap$ and $\beta q-bp$ have opposite signs, because $Q$ lies between $A$ and $B$~: therefore
$$a'+b'=|\alpha q - ap| + |\beta q - bp|=|(\alpha q-ap)-(\beta q-bp)|=|(\alpha-\beta)q-(a-b)p|=Y\wedge Q~;$$ 
$$|a'-b'|=||\alpha q-ap|-|\beta q - bp||=|(\alpha q-ap)+(\beta q-bp)|=|(\alpha+\beta)q-(a+b)p|=X\wedge Q~.$$
For the last identity, compute
\begin{eqnarray*} a'b+b'a
&=&(Q \wedge A)(\infty\wedge B)+(Q \wedge B)(\infty\wedge A) \\
&=& b|q\alpha-pa|+a|q\beta-pb| \\ 
&=& |b(q\alpha-pa)-a(q\beta-pb)| \\ 
&=& q|b\alpha-a\beta|=q(A\wedge B)=q~.
\end{eqnarray*}
\end{proof}

An easy consequence is that all of $a,a',b,b',x,x',y,y'$ are integers of $\llbracket 1,q-1\rrbracket$. Note that the properties of Proposition \ref{prop:abab} are invariant under the exchange of $(a,a')$ with $(b,b')$ and under the exchange of $(a,b,x,y)$ with $(a',b',y',x')$ (which actually amounts to swapping $Q$ and $\infty$).

\begin{proposition} \label{prop:notq2}
None of $a,a',b,b'$ is equal to $\frac{q}{2}$.
\end{proposition}
\begin{proof}
Suppose $b=\frac{q}{2}$. Since $a'b+b'a=q$, we then have $a'=1$. We have $b'a=q-a'b=\frac{q}{2}$ so $a$ divides $\frac{q}{2}$, but $a$ is also coprime to $b=\frac{q}{2}$ (because $A\wedge B=1$). Therefore $a=1$ (which by the way means $A\in\{0,1\}$). But since $a'=1$, this implies that $A$ is a Farey neighbor both of $Q$ and $\infty$, i.e.~$Q$ has the form $\frac{1}{q}$ or $\frac{q-1}{q}$, which we ruled out in the first place.

If instead of $b$ another term of $a,a',b,b'$ is equal to $\frac{q}{2}$, then we can apply the same argument, up to permuting $a,a',b,b'$.
\end{proof}

Notice, however, that one of $a,a',b,b'$ could be \emph{larger} than $\frac{q}{2}$.

\subsection{Level curves on the torus} \label{sec:tori}

Let $\mathbb{T}:=(\mathbb{R}/2\pi\mathbb{Z})^2$ be the standard torus and $\iota:\mathbb{T}\rightarrow \mathbb{C}^2\simeq\mathbb{R}^4$ denote the standard injection, satisfying $$\iota(u,v)=(\cos u, \sin u, \cos v, \sin v).$$ The subgroup $\Gamma=\{\tau_k=(k\frac{2\pi}{q}, kp\frac{2\pi}{q})\}_{k\in \mathbb{Z}}$ of $\mathbb{T}$ lifts to an affine lattice $\Lambda$ of the universal cover $\mathbb{R}^2$ of $\mathbb{T}$. The index of $2\pi\mathbb{Z}^2$ in $\Lambda$ is $q$. Rationals $A,B$ are still as in Theorem \ref{thm:main}.

\begin{proposition} \label{prop:basis}
Define the lifts $u=(a\frac{2\pi}{q},ap\frac{2\pi}{q}-2\alpha\pi)$ and $v=(b\frac{2\pi}{q},bp\frac{2\pi}{q}-2\beta\pi)$ of $\tau_a$ and $\tau_b$ respectively. Also define the center $\overline{c}:=\frac{1}{2}(u+v)$ of the parallelogram $D:=(0,u,u+v,v)$ of $\mathbb{R}^2$. Then $(u,v)$ is a basis of the lattice $\Lambda$, and $D$ is contained in the square $\overline{c}+(-\pi,\pi)^2$.
\end{proposition}
\begin{proof}
Clearly, $\Lambda \subset \mathbb{R}^2$ has covolume $(2\pi)^2/q$. On the other hand, the determinant of $(u,v)$ is $2\pi\frac{2\pi}{q}(\alpha b-a\beta)=\pm(2\pi)^2/q$, so $(u,v)$ is a basis of $\Lambda$.

The abscissae of $u,v$ are clearly positive, and their sum is $\frac{a+b}{q}2\pi=\frac{x}{q}2\pi<2\pi$.

The ordinates $2\pi a(Q-A)$ of $u$ and $2\pi b(Q-B)$ of $v$ have opposite signs, and the sum of their absolute values is $$2\pi \left ( \left | \frac{ap-\alpha q}{q} \right |+ \left |\frac{bp-\beta q}{q}\right | \right )= 2\pi \frac{A\wedge Q + B\wedge Q}{q}=2\pi\frac{y'}{q}<2\pi~,$$
by Proposition \ref{prop:abab}. This proves the claim on $D$. \end{proof}

\begin{definition} \label{def:center} Let $c=\left (\frac{a+b}{q}\pi,[p\frac{a+b}{q} - (\alpha+\beta)]\pi \right )$ denote the projection of $\overline{c}$ to the torus $\mathbb{T}=(\mathbb{R}/2\pi\mathbb{Z})^2$. \end{definition}

\begin{proposition} \label{prop:convexmiss}
Let $\Lambda\subset\mathbb{R}^2$ be a lattice and $P$ be a strictly convex, compact region of $\mathbb{R}^2$ such that $\Lambda\cap \partial P$ consists of the four vertices of a fundamental parallelogram of $\Lambda$. Then $\Lambda\cap P=\Lambda\cap \partial P$ (i.e.~$P$ contains no other lattice points). 
\end{proposition}
\begin{proof}
Without loss of generality, $\Lambda=\mathbb{Z}^2$ and $\{0,1\}^2\subset \partial P$. Since $P$ is strictly convex, the horizontal axis $\mathbb{R}\times \{0\}$ intersects $P$ precisely along $[0,1]\times \{0\}$. A similar statement holds for each side of the unit square. Therefore $P\smallsetminus \{0,1\}^2\subset (0,1)\times \mathbb{R} \cup \mathbb{R}\times (0,1)$, which contains no other vertices of $\mathbb{Z}^2$.
\end{proof}

The idea of the proof of Theorem \ref{thm:main} is to consider a linear form $\rho:\mathbb{R}^4\rightarrow \mathbb{R}$ that takes the same value $Z>0$ on $x_0, x_a, x_b,x_{a+b}$ and check that $\rho<Z$ on all other $x_i$. This will be achieved by looking at the level curve $\gamma$ of $\rho\circ\iota$ in $\mathbb{T}$, of level $Z$, and checking that the lift of $\gamma$ to $\mathbb{R}^2$ bounds a convex body that satisfies the hypotheses of Proposition \ref{prop:convexmiss}. For this, we will need the following property and its corollary.

\begin{lemma} \label{lem:convex}
If $(U,U'),(V,V')\in \mathbb{R}^2\smallsetminus\{(0,0)\}$ and $Z\in \mathbb{R}_+^*$ satisfies $$\left |\sqrt{V^2+V'^2}-\sqrt{U^2+U'^2}\right | <Z<\sqrt{V^2+V'^2}+\sqrt{U^2+U'^2}~,$$ then the preimage of $Z$ under
$$\begin{array}{rrcl} &\mathbb{R}^2 & \rightarrow & \mathbb{R} \\
\psi~: & (x,y) & \longmapsto & (U \cos x + U'\sin x) + (V\cos y+V'\sin y) \end{array}$$
consists of a convex curve $\gamma$ (i.e.~a closed curve bounding a strictly convex domain), together with all the translates of $\gamma$ under $2\pi\mathbb{Z}^2$, which are pairwise disjoint.
\end{lemma}
\begin{proof}
Up to shifting $x$ and $y$ by constants, we can assume $U'=V'=0$ and $U,V>0$. 
Up to exchanging $x$ and $y$, we can furthermore assume $V\geq U$, so that $0\leq V-U<Z<V+U$ and $\psi(x,y)=U\cos x + V\cos y$. Notice that $U,V,Z$ now satisfy all three strong triangular inequalities.

Let $C$ be the square $[-\pi,\pi]^2$. Let us first determine that $\gamma:=\psi^{-1}(Z)\cap C$ is a convex curve contained in the interior of $C$. If $(x,y)\in\gamma$ then $U \cos x \geq Z-V \in (-U,U)$ so $$|x|\leq \arccos \frac{Z-V}{U}\in (0,\pi)~\text{ and }~\pm y=f(x):=\arccos \frac{Z-U \cos x}{V}\in[0,\pi)~,$$
since $Z-U>-V$. Clearly, $f$ vanishes at $\pm \arccos \frac{Z-V}{U}$. Moreover, using the chain rule $(\arccos \circ\, g)''=-\frac{g''(1-g^2)+gg'^2}{(1-g^2)^{3/2}}$, computation yields
$$f''(x)=\frac{-U^2Z}{[V^2-(Z-U\cos x)^2]^{\frac{3}{2}}} \left [ 1+ \frac{V^2-Z^2-U^2}{UZ} \cos x + \cos^2 x\right ]$$
so to show $f''<0$ it is enough to check that the discriminant of the polynomial in $\cos x$ (in the right factor) is negative. This amounts to $\left |\frac{V^2-Z^2-U^2}{UZ}\right |<2$, which in turn follows from the triangular inequalities $(Z+U)^2>V^2$ and $(Z-U)^2<V^2$.

We have proved that $\gamma$ is a convex curve (the union of the graphs of $f$ and $-f$) contained in the interior of $C$: the rest of the lemma follows easily.
\end{proof}

\begin{corollary} \label{cor:patate}
Under the assumptions of Lemma \ref{lem:convex}, 
the set $H:=\psi^{-1}[Z,+\infty)$ consists of the disjoint union of all the convex domains bounded by $\gamma$ and its translates.
\end{corollary}
\begin{proof}
Again restricting to $U'=V'=0<U\leq V$, we see that $H\cap C$ contains the origin (encircled by $\gamma$, and where $\psi$ achieves its maximum $U+V$) and does not contain $(\pi,\pi)$ (where $\psi$ achieves its minimum $-U-V$). The theorem of intermediate values allows us to conclude.
\end{proof}

\section{Proof of Theorem \ref{thm:main}} \label{sec:proof}

Identifying $\mathbb{C}^2$ with $\mathbb{R}^4$ in the standard way, the matrix with column vectors $x_0, x_a,x_b,x_{a+b}$ is
\begin{equation} \label{eq:matrixm} M:= \left ( \begin{array}{clll}
1 & \cos a \frac{2\pi}{q} & \cos b \frac{2\pi}{q} & \cos (a+b) \frac{2\pi}{q} \\ 
0 & \sin a \frac{2\pi}{q} & \sin b \frac{2\pi}{q} & \sin (a+b) \frac{2\pi}{q} \\ 
1 & \cos pa\frac{2\pi}{q} & \cos pb\frac{2\pi}{q} & \cos p(a+b)\frac{2\pi}{q} \\ 
0 & \sin pa\frac{2\pi}{q} & \sin pb\frac{2\pi}{q} & \sin p(a+b) \frac{2\pi}{q}
\end{array} \right )~. \end{equation}
We refer to $\{x_0,x_a,x_b,x_{a+b}\}$ as our \emph{candidate face}.

\subsection{Candidate faces are non-degenerate} \label{sec:nondeg}
\begin{proposition} \label{prop:invertible}
The determinant $D$ of the matrix $M$ is nonzero.
\end{proposition}
\begin{proof}
Rotating the plane of the first two coordinates by $\frac{-a-b}{q}\pi$, and the plane of the last two coordinates by $\frac{-a-b}{q}p\pi$, we see that
\begin{eqnarray*} D&=&\left | \begin{array}{llll}
\cos\frac{-a-b}{q}\pi & \cos\frac{a-b}{q}\pi & 
\cos\frac{ b-a}{q}\pi & \cos\frac{a+b}{q}\pi  \\ 
\sin\frac{-a-b}{q}\pi & \sin\frac{a-b}{q}\pi & 
\sin\frac{ b-a}{q}\pi & \sin\frac{a+b}{q}\pi  \\ 
\cos\frac{-a-b}{q}p\pi & \cos\frac{a-b}{q}p\pi & 
\cos\frac{ b-a}{q}p\pi & \cos\frac{a+b}{q}p\pi  \\ 
\sin\frac{-a-b}{q}p\pi & \sin\frac{a-b}{q}p\pi & 
\sin\frac{ b-a}{q}p\pi & \sin\frac{a+b}{q}p\pi  
\end{array} \right | \\ &=& 4 \left | \begin{array}{llll}
\cos\frac{a+b}{q}\pi & \cos\frac{a-b}{q}\pi & 
\cos\frac{b-a}{q}\pi & \cos\frac{a+b}{q}\pi  \\ 0 & 0 & 
\sin\frac{b-a}{q}\pi & \sin\frac{a+b}{q}\pi  \\ 
\cos\frac{a+b}{q}p\pi & \cos\frac{a-b}{q}p\pi & 
\cos\frac{b-a}{q}p\pi & \cos\frac{a+b}{q}p\pi  \\ 0 & 0 & 
\sin\frac{b-a}{q}p\pi & \sin\frac{a+b}{q}p\pi  
\end{array} \right | \hspace{6pt} \text{(column operations)}\\
&=& 4 \left | \begin{array}{ll}
\cos\frac{a+b}{q} \pi & \cos\frac{a-b}{q} \pi \\
\cos\frac{a+b}{q}p\pi & \cos\frac{a-b}{q}p\pi  
\end{array} \right | \cdot \left | \begin{array}{ll}
\sin\frac{a-b}{q} \pi & \sin\frac{a+b}{q} \pi \\
\sin\frac{a-b}{q}p\pi & \sin\frac{a+b}{q}p\pi  
\end{array} \right | \\ &=& \textstyle{4~(2 
\cos\frac{a}{q} \pi \cos\frac{b}{q} \pi \cdot
\sin\frac{ap}{q}\pi \sin\frac{bp}{q}\pi - 2 
\sin\frac{a}{q} \pi \sin\frac{b}{q} \pi \cdot 
\cos\frac{ap}{q}\pi \cos\frac{bp}{q}\pi)} \\ 
&& \textstyle{~(2
\sin\frac{a}{q} \pi \cos\frac{b}{q} \pi \cdot
\sin\frac{bp}{q}\pi \cos\frac{ap}{q}\pi - 2 
\sin\frac{b}{q} \pi \cos\frac{a}{q} \pi \cdot 
\sin\frac{ap}{q}\pi \cos\frac{bp}{q}\pi)}~,
\end{eqnarray*}
so we only need to prove
$$\tan\frac{ap\pi}{q}\tan\frac{bp\pi}{q} \neq
\tan\frac{a\pi}{q} \tan\frac{b\pi}{q} \hspace{10pt} ; \hspace{10pt}
\tan\frac{a\pi}{q} \tan\frac{bp\pi}{q} \neq
\tan\frac{b\pi}{q} \tan\frac{ap\pi}{q}$$
(provided all these tangents are finite).
Since $ap-\alpha q=a'\cdot \sigma(Q-A)$ (where $\sigma$ is the sign function) and $\tan$ is $\pi$-periodic,
$$\tan \frac{ap\pi}{q}=\tan \frac{ap-\alpha q}{q}\pi=\sigma(Q-A)\tan \frac{a'}{q}\pi$$
and similarly $\tan \frac{bp\pi}{q}=\sigma(Q-B)\tan \frac{b'}{q}\pi$. Since $Q$ lies between $A$ and $B$, the signs of $Q-A$ and $Q-B$ are opposite, so we only need to prove 
\begin{equation} \label{eq:ineq}
\tan\frac{a'\pi}{q}\tan\frac{b'\pi}{q} \neq
-\tan\frac{a\pi}{q} \tan\frac{b\pi}{q} \hspace{10pt} ; \hspace{10pt}
\tan\frac{a\pi}{q} \tan\frac{b'\pi}{q} \neq
-\tan\frac{b\pi}{q} \tan\frac{a'\pi}{q}~.
\end{equation}
(All these tangents \emph{are} finite, by Proposition \ref{prop:notq2}.)
If $a,a',b,b'\leq \frac{q}{2}$, then all the values of ``$\text{tan}$'' in (\ref{eq:ineq}) are positive, which yields the result.

If one of $a,a',b,b'$ is larger than $\frac{q}{2}$, say $b>\frac{q}{2}$, then $a'b+b'a=q$ requires $a'=1$, which entails $a\geq 2$ (because $A$ is not a Farey neighbor of both $Q$ and $\infty$), and $b'\geq 2$ (because $A$ and $B$ are not both Farey neighbors of $Q$). We have $ab'=q-b<\frac{q}{2}$ and $b=\frac{q-ab'}{a'}=q-ab'$.
Therefore the first inequality of (\ref{eq:ineq}) can be written
$$\tan\frac{\pi}{q} \tan\frac{b'\pi}{q} \neq \tan\frac{a\pi}{q} \tan\frac{ab'\pi}{q}~,$$
which is clearly true (both members are positive, but the right one is larger, factor-wise, because $a\geq 2$).

Similarly, the second inequality of (\ref{eq:ineq}) becomes
$\tan\frac{a\pi}{q} \tan\frac{b'\pi}{q} \neq
\tan\frac{ab'\pi}{q} \tan\frac{\pi}{q}$ (all values of ``$\tan$'' are still positive), i.e.
$$\frac{\tan\frac{a\pi}{q}}{\tan\frac{\pi}{q}}  \neq
\frac{\tan\frac{ab'\pi}{q}}{\tan\frac{b'\pi}{q}}~.$$
Notice that without the ``$\tan$'s'', this would be an identity. To see that the right member is larger, it is therefore enough to make sure that the function 
$g:u\mapsto \frac{\tan u}{\tan (u/a)}$ is increasing on $(0,\frac{\pi}{2})$. Computation yields $$g'(u)=\frac{\sin (2 u/a)-\sin (2u)/a}{2\sin^2 (u/a)\cos^2 u}~:$$ since $a\geq 2$, the numerator is clearly positive, by strict concavity of $\sin$ on $[0,\pi]$.

If instead of $b$ another term of $a,a',b,b'$ is larger than $\frac{q}{2}$, then we can apply the same argument, up to permuting $a,a',b,b'$.
\end{proof}

\subsection{Candidate faces are faces of the convex hull}

We must now show that if $\rho:\mathbb{R}^4\rightarrow \mathbb{R}$ is some linear form that takes the same value $Z>0$ 
on each column vector $x_0,x_a,x_b,x_{a+b}$ (i.e.~$\iota(\tau_0),\iota(\tau_a),\iota(\tau_b),\iota(\tau_{a+b})$) of the matrix $M$ from (\ref{eq:matrixm}), then  
$\rho\circ\iota(\tau_k)<Z$ 
for any $k\in \llbracket 0,q-1\rrbracket \smallsetminus \{0,a,b,a+b\}$. This will be done by showing \emph{via} Corollary \ref{cor:patate} that $(\rho\circ\iota)^{-1}[Z,+\infty)$ is (once lifted to $\mathbb{R}^2$) a convex region of the type seen in Proposition \ref{prop:convexmiss}. 

An elementary computation shows that in coordinates, 
\begin{equation} \label{eq:formvalue}\left \{ \begin{array}{rcl}
\rho&=& (-1)^{\alpha+\beta} \left ( \begin{array}{r} 
-\cos \frac{a+b}{q} \pi \sin \frac{ap \pi}{q} \sin \frac{bp \pi}{q} \\
-\sin \frac{a+b}{q} \pi \sin \frac{ap \pi}{q} \sin \frac{bp \pi}{q} \\
\cos \frac{a+b}{q}p \pi \sin \frac{a \pi}{q} \sin \frac{b \pi}{q} \\
\sin \frac{a+b}{q}p \pi \sin \frac{a \pi}{q} \sin \frac{b \pi}{q}
\end{array} \right )^t 
=:\left (\begin{array}{l} U\\ U'\\ V\\ V'\end{array}\right )^t\\ &&\\ Z&=& 
(-1)^{\alpha+\beta} \left ( \cos \frac{a+b}{q}p \pi \sin \frac{a \pi}{q} \sin \frac{b \pi}{q} - \cos \frac{a+b}{q} \pi \sin \frac{ap \pi}{q} \sin \frac{bp \pi}{q} \right ) \\
&=& \frac{(-1)^{\alpha+\beta}}{2} \left ( \cos \frac{a+b}{q}p \pi \cos \frac{a-b}{q} \pi - \cos \frac{a+b}{q} \pi \cos \frac{a-b}{q}p \pi \right )\end{array} \right .\end{equation}
will do ($Z$ will turn out to be positive by Claim \ref{cla:key} below; so far we only know $Z\neq 0$ by Proposition \ref{prop:invertible}). The notation $U,U',V,V'$ is made to fit Lemma \ref{lem:convex}. Define
$$\left \{ \begin{array}{rclcl}
U''&:=&\sqrt{U^2+U'^2}&=& |\sin \frac{ap \pi}{q} \sin \frac{bp \pi}{q}|>0 \\
V''&:=&\sqrt{V^2+V'^2}&=& |\sin \frac{a \pi}{q} \sin \frac{b \pi}{q}|>0~.
\end{array}\right .$$

\begin{claim} \label{cla:key}
The point $c$ of Definition \ref{def:center} is the absolute maximum of $\rho \circ \iota$ on the torus $\mathbb{T}$. Moreover, $$Z=\cos \frac{x'}{q}\pi\cos\frac{y}{q}\pi - \cos \frac{x}{q}\pi \cos \frac{y'}{q}\pi~,$$ $Z$ is positive, and one has: $|V''-U''|<Z<V''+U''$.
\end{claim}

This claim proves Theorem \ref{thm:main}. Indeed, assume the claim, and let $H$ denote $[Z,+\infty)$. Let $\overline{\pi}$ denote the natural projection $\mathbb{R}^2\rightarrow \mathbb{T}$. By Corollary \ref{cor:patate}, the level curve $(\rho\circ\iota\circ \overline{\pi})^{-1}(Z)\subset \mathbb{R}^2$ contains a striclty convex closed curve $\gamma$ centered around $\overline{c}$, contained in the square $C:=\overline{c}+(-\pi,\pi)^2$ and passing through the representatives of $\tau_0,\tau_a,\tau_b,\tau_{a+b}$ contained in $C$. By Proposition \ref{prop:basis}, these representatives are the vertices $0,u,v,u+v$ of the fundamental parallelogram $D$. Corollary \ref{cor:patate} and Proposition \ref{prop:convexmiss} then yield the result: $(\rho\circ\iota)^{-1}(H)$ contains no other points $\tau_k$ than $\tau_0,\tau_a,\tau_b,\tau_{a+b}$.

\begin{proof} (Claim \ref{cla:key}). The maximum of $\rho \circ \iota$ on $\mathbb{T}$ is clearly $U''+V''$. Since $$\iota(c)=\left ( \begin{array}{c} \cos \frac{a+b}{q}\pi \\ \sin \frac{a+b}{q}\pi \\ \cos [p\frac{a+b}{q} - (\alpha+\beta)]\pi \\ \sin [p\frac{a+b}{q} - (\alpha+\beta)]\pi \end{array} \right )~,$$ we can compute 
\begin{eqnarray*} \rho\circ\iota(c)&=&(-1)^{\alpha+\beta}\left (-\sin\frac{ap\pi}{q}\pi\sin\frac{bp\pi}{q}\pi+(-1)^{\alpha+\beta}\sin\frac{a\pi}{q}\sin\frac{b\pi}{q} \right ) \\ &=& -\sin\frac{ap-\alpha q}{q}\pi\sin\frac{bp-\beta q}{q}\pi + \sin\frac{a}{q}\pi\sin\frac{b}{q}\pi \\ 
&=& \sin\frac{A\wedge Q}{q}\pi\sin\frac{B\wedge Q}{q}\pi + \sin\frac{a}{q}\pi\sin\frac{b}{q}\pi \\
&=& \sin\frac{a'}{q}\pi\sin\frac{b'}{q}\pi + \sin\frac{a}{q}\pi\sin\frac{b}{q}\pi \end{eqnarray*}
because $ap-\alpha q$ and $bp-\beta q$ have opposite signs ($Q$ lies between $A$ and $B$). Both terms in the last expression are positive since $a,a',b,b'\in \llbracket 1,q-1 \rrbracket$. In fact, since $$V''= \left | \sin \frac{a \pi}{q} \sin \frac{b \pi}{q}\right |=\sin \frac{a \pi}{q} \sin \frac{b \pi}{q}$$ and $$U''= \left |\sin \frac{a p\pi}{q} \sin \frac{bp \pi}{q}\right |=\left |\sin \frac{a p-\alpha q}{q}\pi \sin \frac{bp -\beta q}{q}\pi\right |=\sin \frac{a' \pi}{q} \sin \frac{b' \pi}{q}~,$$ we have shown that $\rho\circ \iota(c)=U''+V''$, the absolute maximum of $\rho\circ\iota$.

The computation of $Z$ follows similar lines: in the second expression for $Z$ in (\ref{eq:formvalue}), notice that the first and last cosines can be written $$(-1)^{\alpha+\beta}\cos\frac{a+b}{q}p\pi=\cos\frac{(a+b)p-(\alpha+\beta)q}{\pi}=\cos \frac{X\wedge Q}{q}\pi~;$$ $$(-1)^{\alpha-\beta}\cos\frac{a-b}{q}p\pi=\cos\frac{(a-b)p-(\alpha-\beta)q}{\pi}=\cos \frac{Y\wedge Q}{q}\pi$$ (using Proposition \ref{prop:abab}). Together with $\frac{a+b}{q}=\frac{x}{q}$ and $\frac{a-b}{q}=\frac{\pm y}{q}$, this yields the desired expression of $Z=\cos \frac{x'}{q}\pi\cos\frac{y}{q}\pi - \cos \frac{x}{q}\pi \cos \frac{y'}{q}\pi$.

The upper bound on $Z$ is obvious from the first expression of $Z$ in (\ref{eq:formvalue}). We now focus on the lower bound (which will also imply $Z>0$), i.e.~we aim to show
\begin{equation} \label{eq:minoration} 
\cos\frac{x'}{q}\pi \cdot \cos\frac{y}{q}\pi -
\cos\frac{x}{q}\pi  \cdot \cos\frac{y'}{q}\pi > 2\left | 
\sin\frac{a'}{q}\pi \cdot \sin\frac{b'}{q}\pi - 
\sin\frac{a}{q}\pi  \cdot \sin\frac{b}{q}\pi \right |~. 
\end{equation}
By Proposition \ref{prop:abab}, the right member of (\ref{eq:minoration}) can be written
$$\left | \left (
\cos\frac{x'}{q}\pi-
\cos\frac{y'}{q}\pi\right )-\left (
\cos\frac{y}{q}\pi-
\cos\frac{x}{q}\pi\right )\right |~;$$
therefore we are down to proving the two identities

$$\left \{ \begin{array}{l}\displaystyle{
\left (\cos \frac{x'}{q}\pi-1 \right )\cdot 
\left (\cos \frac{y}{q}\pi +1 \right )~>~
\left (\cos \frac{x}{q}\pi +1 \right )\cdot 
\left (\cos \frac{y'}{q}\pi-1 \right )} \\ \\ \displaystyle{
\left (\cos \frac{x'}{q}\pi+1 \right )\cdot 
\left (\cos \frac{y}{q}\pi -1 \right )~>~
\left (\cos \frac{x}{q}\pi -1 \right )\cdot 
\left (\cos \frac{y'}{q}\pi+1 \right )~.}
\end{array} \right . $$

\noindent Using $\cos t +1=2\cos^2\frac{t}{2}$ and $\cos t - 1=-2\sin^2\frac{t}{2}$, this in turn amounts to

$$\left \{\begin{array}{l} \displaystyle{
\sin \left (\frac{x'}{q}\right )\frac{\pi}{2}\cdot
\cos \left (\frac{y}{q}\right )\frac{\pi}{2}~<~
\cos \left (\frac{x}{q}\right )\frac{\pi}{2}\cdot 
\sin \left (\frac{y'}{q}\right )\frac{\pi}{2}} \\ \\ \displaystyle{ 
\cos \left (\frac{x'}{q}\right )\frac{\pi}{2}\cdot
\sin \left (\frac{y}{q}\right )\frac{\pi}{2}~<~
\sin \left (\frac{x}{q}\right )\frac{\pi}{2}\cdot 
\cos \left (\frac{y'}{q}\right )\frac{\pi}{2}~,}
\end{array} \right .$$

\noindent or equivalently
\begin{equation}\label{eq:sineratio}\left \{\begin{array}{rcll}
\frac{\displaystyle{\sin \left (\frac{x'}{q}\right )\frac{\pi}{2}}}
{\displaystyle{\sin \left (\frac{y'}{q}\right )\frac{\pi}{2}}} &<&
\frac{\displaystyle{\sin \left (\frac{q - x}{q}\right )\frac{\pi}{2}}}
{\displaystyle{\sin \left (\frac{q - y}{q}\right )\frac{\pi}{2}}} 
& \hspace{20pt} (i) \\  \\ 
\frac{\displaystyle{\sin \left (\frac{y}{q}\right )\frac{\pi}{2}}}
{\displaystyle{\sin \left (\frac{x}{q}\right )\frac{\pi}{2}}} &<&  
\frac{\displaystyle{\sin \left (\frac{q-y'}{q}\right )\frac{\pi}{2}}}
{\displaystyle{\sin \left (\frac{q-x'}{q}\right )\frac{\pi}{2}}}
& \hspace{20pt} (ii). \end{array} \right .\end{equation}
To prove (\ref{eq:sineratio})-$(i)$ and (\ref{eq:sineratio})-$(ii)$, we will use
\begin{proposition} \label{prop:sines}
If $\displaystyle{0<s<t<\frac{\pi}{2}}$ and $\displaystyle{0<s'<t'<\frac{\pi}{2}}$ satisfy $s<s'$ and $\displaystyle{\frac{s}{t}\leq\frac{s'}{t'}}$, then $\displaystyle{\frac{\sin s}{\sin t}<\frac{\sin s'}{\sin t'}}$.
\end{proposition}
\begin{proof}
Up to decreasing $t$, it is clearly enough to treat the case $\frac{s}{t}=\frac{s'}{t'}=\frac{1-\lambda}{1+\lambda}$ (where $0<\lambda<1$). The result then follows from the fact that $f(u)=\frac{\sin(1-\lambda)u}{\sin(1+\lambda)u}$ is increasing on $(0,\frac{\pi}{2(1+\lambda)}]$, which can be seen by computing
$$ f'(u)=\frac{\sin(2\lambda u)-\lambda\sin(2u)}{\sin^2(1+\lambda)u}~:$$
here the numerator is positive by strong concavity of $\sin$ on $[0,\frac{\pi}{1+\lambda}]$.
\end{proof}

We now prove (\ref{eq:sineratio})-$(i)$: by Proposition \ref{prop:sines}, it is enough to check
$$0<x'<y'< q\text{ and }0<y<x<q$$ (which are obvious from Proposition \ref{prop:abab}), plus
\begin{equation}\label{eq:enfin} 
x' < q-x 
~\text{ and }~ 
\frac{x'}{y'}\leq 
\frac{q-x}{q-y}~.
\end{equation}
The first inequality of (\ref{eq:enfin}) amounts, by Proposition \ref{prop:abab}, to
$$|a-b|+(a+b)<a'b+b'a$$
which can be written
$$(a'-1)(b\pm 1)+(b'-1)(a\mp 1)>0~.$$
If $a'$ and $b'$ are $>1$, then at least one of the products in the left member is positive, and we are done. If $a'=1$, then $b'>1$ (because $A,B$ are not both Farey neighbors of $Q$ in the assumptions of Theorem \ref{thm:main}) and $a>1$ (because $Q,\infty$ have no common Farey neighbors, i.e.~$p\notin\{1,q-1\}$) and we are also done. If $b'=1$, the argument is the same, exchanging $(A,a,a')$ and $(B,b,b')$.

The second inequality of (\ref{eq:enfin}) amounts to
$$q(y'-x')\geq y'x-x'y$$
which by Proposition \ref{prop:abab} can also be written
$$ y'-x' \geq \frac{(a'+b')(a+b)-|(a'-b')(a-b)|}{a'b+b'a}=:H~.$$
Here the left member is at least 2~: indeed, by Proposition \ref{prop:abab} it can be written $$a'+b'-|a'-b'|=2\inf\{a',b'\}~.$$
The right member $H$, however, is at most $2$~: indeed,
\begin{eqnarray*}2-H 
&=&\frac{a'(2b-a-b)+b'(2a-a-b)+|(a'-b')(a-b)|}{a'b+b'a} \\
&=& \frac{(a'-b')(b-a)+|(a'-b')(a-b)|}{a'b+b'a}
\end{eqnarray*}
and the numerator has the form $u+|u|\geq 0$. This finishes the proof of (\ref{eq:sineratio})-$(i)$.

The proof of (\ref{eq:sineratio})-$(ii)$ is identical with that of (\ref{eq:sineratio})-$(i)$, exchanging $(a,b,x,y)$ with $(a',b',y',x')$. Claim \ref{cla:key}, and therefore Theorem \ref{thm:main}, are proved. 
\end{proof}

\section{General finite subgroups of $\iota(\mathbb{T})\subset{\mathbb{C}^*}^2$}
\label{sec:general}
In this last section, let $\Gamma$ be \emph{any} finite subgroup of $\mathbb{T}=(\mathbb{S}^1)^2=(\mathbb{R}/2\pi\mathbb{Z})^2$. There exists a unique rational $Q=\frac{p}{q}\in[0,1)$ (here in reduced form) and a unique pair $(\mu,\nu)\in\mathbb{Z}_{>0}^2$ such that $\Gamma$ is the preimage of $\{\tau_k=(\frac{k}{q},\frac{kp}{q})\}_{0\leq k < q}$ under $$\begin{array}{rrcl} & \mathbb{T}& \longrightarrow & \mathbb{T} \\ \psi_{\mu\nu}~: & (s,t) & \mapsto & (\mu s,\nu t)~. \end{array}$$
Indeed, $\mu$ (resp. $\nu$) is just the cardinality of $\Gamma\cap(\mathbb{S}^1\times\{0\})$ (resp. $\Gamma\cap (\{0\}\times\mathbb{S}^1)$); the order of $\Gamma$ is $q\mu\nu$. The case $\frac{p}{q}=0$ can be put aside: it corresponds to $\iota(\Gamma)\subset\mathbb{R}^4$ being (the vertices of) the Cartesian product of a regular $\mu$-gon with a regular $\nu$-gon (the 3-dimensional faces are then regular prisms; degeneracies occur if $\mu\leq 2$ or $\nu \leq 2$). The case $\mu=\nu=1$ was treated in the previous sections, including the discussion of degeneracies when $p\equiv 0 \text{ or } \pm 1~[\text{mod }q]$.

It is easy to see that if $\mu=1<\nu$ (resp. $\nu=1<\mu$) and $\frac{p}{q}=\frac{1}{2}$, then $\iota(\Gamma)$ is contained in a $3$-dimensional subspace of $\mathbb{R}^4$ --- in fact, $\iota(\Gamma)$ is the vertex set of an antiprism with $\nu$-gonal (resp. $\mu$-gonal) basis, which in turn degenerates to a tetrahedron when $\nu=2$ (resp. $\mu=2$). Therefore, we can make
\begin{assumption} \label{ass:qumunu}
Until the end of this section,
\begin{itemize}
\item at least one of the positive integers $\mu,\nu$ is larger than one;
\item the rational $\frac{p}{q}\in (0,1)$ is not $\frac{1}{2}$ when $\mu=1$ or $\nu=1$.
\end{itemize}
\end{assumption}
Then, we claim that faces of the convex hull of $\iota(\gamma)\subset\mathbb{R}^4$ come in three types:
\begin{enumerate}
\item If $A,B\in[0,1]$ are rationals satisfying the hypotheses of Theorem \ref{thm:main}, then there is a tetrahedron spanned by the images under $\iota:\mathbb{T}\rightarrow \mathbb{R}^4$ of 
$$\textstyle{
\left (\frac{0}{q\mu}, \frac{0}{q\nu} \right ), 
\left (\frac{a}{q\mu}2\pi, \frac{ap-\alpha q}{q\nu}2\pi \right),
\left (\frac{b}{q\mu}2\pi, \frac{bp-\beta q}{q\nu}2\pi \right),
\left (\frac{a+b}{q\mu}2\pi, \frac{(a+b)p-(\alpha+\beta) q}{q\nu}2\pi\right),}$$
which are clearly four points of $\Gamma=\psi_{\mu\nu}^{-1}\{\tau_1,\dots,\tau_q\}$. They form a parallelogram whose center is $c=\left ( \frac{a+b}{q\mu}\pi, \frac{(a+b)p-(\alpha+\beta) q}{q\nu}\pi \right)$.
\item If $\nu>1$, add an extra tetrahedron of the type above for the pair $\{A,B\}=\{\frac{0}{1},\frac{1}{1}\}$ (this was ruled out in Theorem \ref{thm:main} because $A,B$ were not allowed both to be Farey neighbors of $\infty=\frac{1}{0}$). Similarly, if $\mu>1$, add an extra tetrahedron of the type above for $\{A,B\}$ equal to the unique pair of Farey neighbors $\frac{\alpha}{a},\frac{\beta}{b}$ such that $\frac{\alpha+\beta}{a+b}=\frac{p}{q}$. (If $\frac{p}{q}=\frac{1}{2}$ and $\mu,\nu\geq 2$, these two ``extra'' tetrahedra are in fact the same one.)
\item If $\nu>1$, add an extra cell spanned by the $2\nu$ vertices images under $\iota$ of 
$$\textstyle{\left \{\left . \left (0,\frac{k}{\nu}2\pi\right )~\right |~0\leq k < \nu
\left \} \,
\cup \left \{\left . \left (\frac{1}{q\mu}2\pi,\frac{p+kq}{q\nu}2\pi \right )~\right |~0\leq k<\nu\right \}
\right . \right .
.}$$ 
If $\nu>2$, this cell is an antiprism with regular $\nu$-gonal basis; it degenerates to a tetrahedron when $\nu=2$. Similarly, if $\mu>1$, add an extra cell spanned by the $2\mu$ vertices images under $\iota$ of 
$$\textstyle{\left \{\left . \left (\frac{k}{\mu}2\pi,0\right )~\right |~0\leq k < \mu\right \}\cup\left \{\left . \left (\frac{p+kq}{q\mu}2\pi,\frac{1}{q\nu}2\pi \right )~\right |~0\leq k<\mu\right \}~.}$$
\end{enumerate} 

Actually, cells of type (3) degenerate to segments when $\mu,\nu=1$.

\begin{observation} \label{obs:whoisone}
Let $\{A,B\}\subset [0,1]$ be a pair of rationals describing a face of type (1) or (2), define $a,a',b,b'\in\mathbb{Z}_{>0}$ and $x,x',y,y'\in\mathbb{Z}_{\geq 0}$ in the usual way, and bear in mind Proposition \ref{prop:abab}. Then,
\begin{itemize}
\item having $a=b=1$ (i.e.~$y=0$, i.e.~$y'=q$) is only allowed if $\nu>1$; 
\item having $a'=b'=1$ (i.e.~$x'=0$, i.e.~$x=q$) is only allowed if $\mu>1$; \item Proposition \ref{prop:notq2} no longer holds: some of $a,a',b,b'$ may be equal to $\frac{q}{2}$.
\end{itemize}
\end{observation}

First we prove that cells of types (1)--(2)--(3), pushed forward by $\Gamma$, are combinatorially glued face-to-face (i.e.~an analogue of Theorem \ref{thm:nofaces} holds). The proof exactly shadows that of Theorem \ref{thm:nofaces} (lifting to the cover $\psi_{\mu\nu}$), except that when $\mu>1$ (resp. $\nu>1$), we must check that faces of type (2)--(3) also fit together correctly.

Assume $\nu>1$: the ``first'' tetrahedron (of type (2) in the list), corresponding to $\{A,B\}=\{\frac{0}{1},\frac{1}{1}\}$, is spanned (up to action of $\Gamma$) by the images under $\iota$ of 
$$\left (\frac{0}{q\mu},\frac{0}{q\nu}\right ), 
\left (\frac{1}{q\mu}2\pi, \frac{p}{q\nu}2\pi \right ), 
\left (\frac{1}{q\mu}2\pi, \frac{p-q}{q\nu}2\pi \right ), 
\left (\frac{2}{q\mu}2\pi, \frac{2p-q}{q\nu}2\pi\right )~. $$ 
The subfaces obtained by dropping the second or third of these four vertices also belong to faces of type (1) (with $\{A,B\}=\{\frac{0}{1},\frac{1}{2}\}$ or $\{\frac{1}{2},\frac{1}{1}\}$), by the argument of the proof of Theorem \ref{thm:nofaces}. The face obtained by dropping the last vertex is clearly a face of the $\nu$-antiprism of type (3). The face obtained by dropping the first vertex is clearly a face of that same antiprism, shifted by $(\frac{1}{q\mu},\frac{p}{q\nu})\in\Gamma$. The antiprism and its shift, finally, are glued base-to-base along $\iota\left \{(\frac{1}{q\mu}2\pi,\frac{p+kq}{q\nu}2\pi)~|~0\leq k<\nu\right \}$.

A similar argument holds when $\mu>1$ near the ``end'' of the sequence of tetrahedra: again, this just amounts to swapping $Q$ and $\infty$.

Next, we proceed to show that the candidate faces of types (1)--(2)--(3) are indeed faces of the convex hull of $\iota(\Gamma)$.

\subsection{Faces of type (3)} The vertices
$\left ( \!\! \begin{array}{c} 1 \\ 0 \\ \cos \frac{2k\pi}{\nu} \\ \sin \frac{2k\pi}{\nu} \end{array} \!\! \right )_{\!0\leq k<\nu}$ and 
$\left ( \!\! \begin{array}{c} \cos \frac{2\pi}{q\mu} \\ \sin \frac{2\pi}{q\mu} \\ \cos  \frac{2\pi (p+kq)}{q\nu} \\ \sin  \frac{2\pi (p+kq)}{q\nu} \end{array} \!\! \right )_{\!0\leq k<\nu}$ 
form two regular $\nu$-gons contained in \emph{distinct} planes parallel to $\{(0,0)\}\times \mathbb{R}^2$, and are not translates of each other (they are off by a rotation of angle $2\pi \frac{p}{q\nu} \notin \frac{2\pi}{\nu}\mathbb{Z}$): this shows that they are the vertices of a convex, non--degenerate antiprism. Moreover, these $2\nu$ vertices clearly maximize the linear form $\rho=(\cos \frac{\pi}{q\mu},\sin\frac{\pi}{q\mu},0,0)$ (that is a purely 2-dimensional statement) and therefore span a face of the convex hull of $\iota(\Gamma)$. Similarly, the vertices of the other antiprism maximize $\rho'=(0,0,\cos\frac{\pi}{q\nu},\sin\frac{\pi}{q\nu})$.

\subsection{Faces of type (1) and (2)} Let $\{A,B\}=\{\frac{\alpha}{a},\frac{\beta}{b}\}$ be as in type (1) or (2); the candidate face now is spanned by the column vectors of
$$M:=\left ( \begin{array}{cccc}
1&\cos\frac{a}{\mu q}2\pi&\cos\frac{b}{\mu q}2\pi&\cos\frac{a+b}{\mu q}2\pi \\
0&\sin\frac{a}{\mu q}2\pi&\sin\frac{b}{\mu q}2\pi&\sin\frac{a+b}{\mu q}2\pi \\
1&\cos\frac{ap-\alpha q}{\nu q}2\pi&\cos\frac{bp-\alpha q}{\nu q}2\pi&
\cos\frac{(a+b)p-(\alpha+\beta)q}{\nu q}2\pi \\
0&\sin\frac{ap-\alpha q}{\nu q}2\pi&\sin\frac{bp-\alpha q}{\nu q}2\pi&
\sin\frac{(a+b)p-(\alpha+\beta)q}{\nu q}2\pi
\end{array}\right )~.$$
We now transpose the argument of Section \ref{sec:proof}. Generally speaking, the presence of $\mu,\nu\geq 1$ makes \emph{even more true} any given inequality that we have to check, but we must check it also for the extra tetrahedra of type (2): hence some additional care.
\subsection*{Candidate faces are non-degenerate}
Rotating the first two coordinates by $\frac{-a-b}{\mu q}\pi$ and the last two by $\frac{-(a+b)p+(\alpha+\beta)q}{\nu q}\pi=\frac{-(ap-\alpha q)-(bp-\beta q)}{\nu q}\pi$, using the method of Section \ref{sec:nondeg}, and replacing $\frac{(ap-\alpha q)\pm(bp-\beta q)}{\nu q}$ with $\frac{a'\mp b'}{\nu q}\cdot\sigma(ap-\alpha q)$, compute
$$\begin{array}{rcl} \det M &=& \pm 4
\left | \begin{array}{cc} 
\cos \frac{a+b}{\mu q}\pi& \cos \frac{a-b}{\mu q}\pi \\
\cos \frac{a'-b'}{\nu q} \pi & 
\cos \frac{a'+b'}{\nu q} \pi \end{array} \right | \cdot
\left | \begin{array}{cc} 
\sin \frac{a-b}{\mu q}\pi& \sin \frac{a+b}{\mu q}\pi \\
\sin \frac{a'+b'}{\nu q} \pi & 
\sin \frac{a'-b'}{\nu q} \pi \end{array} \right |\\
&=&\pm 16( 
\cos\frac{a\pi }{\mu q} \cos\frac{b\pi }{\mu q}
\sin\frac{a'\pi}{\nu q} \sin\frac{b'\pi}{\nu q} +  
\sin\frac{a\pi}{\mu q} \sin\frac{b\pi}{\mu q} 
\cos\frac{a'\pi}{\nu q}\cos\frac{b'\pi}{\nu q}) \\ 
&& \cdot \,(
\sin\frac{a\pi}{\mu q} \cos\frac{b\pi}{\mu q}
\sin\frac{b'\pi}{\nu q} \cos\frac{a'\pi}{\nu q} + 
\sin\frac{b\pi}{\mu q} \cos\frac{a\pi}{\mu q} 
\sin\frac{a'\pi}{\nu q} \cos\frac{b'\pi}{\nu q})~.
\end{array}$$
To follow up the method of Section \ref{sec:nondeg}, we would divide both factors of $\det M$ by 
$$\textstyle{H:=\cos \frac{a\pi}{\mu q}\cos \frac{b\pi}{\mu q}\cos \frac{a'\pi}{\nu q}\cos \frac{b'\pi}{\nu q}}~:$$ however, that number can be $0$. In that case, each factor of $\det M$ has a vanishing summand. Let us prove that the other summand is then nonzero, so that $\det M\neq 0$. (Note that the \emph{sines} in $\det M$ never vanish, only the \emph{cosines} may.)

If $\cos\frac{a\pi}{\mu q}=0$, then $\mu=1$ and $a=\frac{q}{2}$. This implies $\nu>1$ by Assumption \ref{ass:qumunu}, so the first factor of $\det M$ has a nonzero second summand. Moreover, the second factor of $\det M$ has a nonzero first summand unless $\cos \frac{b\pi}{\mu q}=0$ i.e.~$b=\frac{q}{2}=a$. But $a,b$ are coprime, so we then have $a=b=1$ and $q=2$ and $\frac{p}{q}=\frac{1}{2}$, which is ruled out when $\mu=1$ (Assumption \ref{ass:qumunu}). If another factor of $H$ vanishes, the argument is similar up to switching $(a,a')$ with $(b,b')$, and/or $(a,b,\mu)$ with $(a',b',\nu)$. In any case, $M$ is invertible. On the other hand, if $H\neq 0$, we must make sure that
\begin{equation} \label{eq:munutan}
\textstyle{\tan\frac{a'\pi}{\nu q} \tan\frac{b'\pi}{\nu q} \neq 
- \tan\frac{a\pi}{\mu q} \tan\frac{b\pi}{\mu q}} ~\text{ ; }~
\textstyle{\tan\frac{a\pi}{\mu q} \tan\frac{b'\pi}{\nu q} \neq
-\tan\frac{b\pi}{\mu q} \tan\frac{a'\pi}{\nu q}}~.
\end{equation}
If $\mu>1$ and $\nu>1$, all tangents in (\ref{eq:munutan}) are positive, so (\ref{eq:munutan}) holds.

Suppose $\mu=1<\nu$. Then at most one of $a',b'$ is equal to $1$ (Observation \ref{obs:whoisone}). If $a,b<\frac{q}{2}$, the members in (\ref{eq:munutan}) have opposite signs. If $a>\frac{q}{2}$, since $a'b+b'a=q$, we have $b'=1$ which implies $a'>1$ and $a=q-a'b$. Thus, (\ref{eq:munutan}) becomes
$$\textstyle{\tan\frac{a'\pi}{\nu q} \tan\frac{\pi}{\nu q} \neq 
\tan\frac{a'b\pi}{q} \tan\frac{b\pi}{q}~\text{ ; }~\left . \tan\frac{a'b\pi}{q} \right / \tan\frac{b\pi}{q} \neq \left . \tan\frac{a'\pi}{\nu q} \right / \tan\frac{\pi}{\nu q}~:}$$
in the first inequality, even if $b=1$, the right member is larger because $\nu>1$. In the second inequality, even if $b=1$, the method of Section \ref{sec:nondeg} shows that the left member is larger because $\nu>1$ and $a'>1$. 

If $b>\frac{q}{2}$, the argument is the same, exchanging $(a,a')$ with $(b,b')$. Finally, if $\nu=1<\mu$, the argument is again the same, switching $(a,b,\mu)$ with $(a',b',\nu)$. Therefore, the matrix $M$ is invertible and the candidate face is non-degenerate.

\subsection*{Candidate faces are faces of the convex hull}
Let us now prove that if a linear form $\rho=(U,U',V,V')$ takes the same value $Z>0$ on each column vector of $M$, then $\rho\circ\iota$ achieves its maximum on $\mathbb{T}$ at $c$ and $|V''-U''|<Z<V''+U''$, where $U''=\sqrt{U^2+U'^2}$ and $V''=\sqrt{V^2+V'^2}$ (by the argument after Claim \ref{cla:key}, this will show that the candidate face is a face of the convex hull). An elementary computation shows that
\begin{equation*} \left \{ \begin{array}{rcl}
\rho&=& \left (\begin{array}{r} 
-\cos \frac{a+b}{\mu q} \pi 
\sin \frac{ap-\alpha q}{\nu q}\pi \sin \frac{bp-\beta q}{\nu q} \pi \\
-\sin \frac{a+b}{\mu q} \pi 
\sin \frac{ap-\alpha q}{\nu q}\pi \sin \frac{bp-\beta q}{\nu q}\pi \\
\cos \frac{(ap-\alpha q)+(bp-\beta q)}{\nu q}\pi 
\sin \frac{a}{\mu q}\pi \sin \frac{b}{\mu q}\pi \\
\sin \frac{(ap-\alpha q)+(bp-\beta q)}{\nu q}\pi 
\sin \frac{a}{\mu q}\pi \sin \frac{b}{\mu q}\pi
\end{array}  \right )^t 
=:\left (\begin{array}{l} U\\ U'\\ V\\ V'\end{array}
\right )^t\\ &&\\ Z&=&
\cos \frac{(ap-\alpha q)+(bp-\beta q)}{\nu q}\pi 
\sin \frac{a \pi}{\mu q} \sin \frac{b \pi}{\mu q} 
-\cos \frac{a+b}{\mu q} \pi 
\sin \frac{ap-\alpha q}{\nu q}\pi 
\sin \frac{bp -\beta q}{\nu q}\pi \\ &=& 
\frac{1}{2} \left ( \cos \frac{x'\pi}{\nu q} \cos \frac{y\pi}{\mu q} 
-\cos \frac{x\pi}{\mu q} \cos \frac{y'\pi}{\nu q} \right ) 
\end{array} \right . \end{equation*}
will do (the second expression of $Z$ follows from the first one and from the fact that $(ap-\alpha q)(bp-\beta q)<0$ --- again, the sign of $Z$ remains to be checked). First, 
$$\begin{array}{rcl}\rho\circ\iota(c)&=&-\sin\frac{ap-\alpha q}{\nu q}\pi\sin\frac{bp-\beta q}{\nu q}\pi+\sin\frac{a}{\mu q}\pi\sin\frac{b}{\mu q}\pi \\ 
&=&\sin\frac{a'}{\nu q}\pi\sin\frac{b'}{\nu q}\pi+\sin\frac{a}{\mu q}\pi\sin\frac{b}{\mu q}\pi=U''+V''\end{array}$$
(again because $(ap-\alpha q)(bp-\beta q)<0$), so $\underset{\mathbb{T}}{\max}(\rho\circ\iota)=\rho\circ\iota(c)$. 

The upper bound $U''+V''$ for $Z$ is clear from its first expression; the lower bound follows lines similar to the proof of Claim \ref{cla:key}: we just need to check
$$\textstyle{2Z=\cos \frac{x'\pi}{\nu q} \cos\frac{y\pi}{\mu q} - 
\cos \frac{x\pi}{\mu q} \cos\frac{y'\pi}{\nu q}>2 \left | 
\sin \frac{a\pi}{\mu q} \sin \frac{b\pi}{\mu q}- 
\sin \frac{a'\pi}{\nu q} \sin \frac{b'\pi}{\nu q} \right |~.}$$
The right member being $|(\cos\frac{x'}{\nu q}\pi-\cos\frac{y'}{\nu q}\pi
)-(\cos\frac{y}{\mu q}\pi-\cos\frac{x}{\mu q}\pi)|$, we only need
$$\textstyle{
(\cos \frac{x'}{\nu q}\pi\pm 1 )\cdot 
(\cos \frac{y}{\mu q}\pi \mp 1 )~>~
(\cos \frac{x}{\mu q}\pi \mp 1 )\cdot 
(\cos \frac{y'}{\nu q}\pi\pm 1 )}$$
which amounts to
\begin{equation}\label{eq:sineratio2} {
\frac{\sin \frac{x'}{\nu q} \cdot \frac{\pi}{2}}
{\sin \frac{y'}{\nu q}\cdot\frac{\pi}{2}} <
\frac{\sin \frac{\mu q - x}{\mu q}\cdot \frac{\pi}{2}}
{\sin \frac{\mu q - y}{\mu q}\cdot \frac{\pi}{2}} 
\hspace{6pt}(i)\hspace{8pt}\text{ and }~ 
\frac{\sin \frac{y}{\mu q} \cdot \frac{\pi}{2}}
{\sin \frac{x}{\mu q} \cdot\frac{\pi}{2}} <  
\frac{\sin \frac{\nu q-y'}{\nu q}\cdot\frac{\pi}{2}}
{\sin \frac{\nu q-x'}{\nu q}\cdot\frac{\pi}{2}}
\hspace{6pt}(ii)~.} \end{equation}
Let us focus on (\ref{eq:sineratio2})-$(i)$. By Proposition \ref{prop:sines},
it is enough to check
$0<x'<y'< \nu q$ and $0<y<x<\mu q$ (which are clear from Proposition \ref{prop:abab}: indeed, by Observation \ref{obs:whoisone}, we \emph{may} have $y'=q$ but then $\nu>1$; we \emph{may} have $x=q$ but then $\mu>1$), plus
\begin{equation}\label{eq:enfin2} 
\textstyle
{
\frac{x'}{\nu} < q-\frac{x}{\mu} ~\text{ and }~ 
\frac{x'}{y'}\leq \frac{\mu q-x}{\mu q-y}~.}\end{equation}
The first inequality of (\ref{eq:enfin2}) can be written
$\frac{|a'-b'|}{\nu}+\frac{a+b}{\mu}<a'b+b'a$, or equivalently,
$$\textstyle{(a'-\frac{1}{\mu}) \cdot (b\pm \frac{1}{\nu})+
(b'-\frac{1}{\mu})\cdot (a\mp \frac{1}{\nu})>0~.}$$
If $\mu,\nu>1$ this is obvious. If $\mu=1<\nu$, then at least one of $a',b'$ is larger than $1$ (Observation \ref{obs:whoisone}), and the product where it appears is positive: done. If $\nu=1<\mu$, then at least one of $a,b$ is larger than one and we are also done.

The second inequality of (\ref{eq:enfin2}) can be written $\mu (y'-x')\geq \frac{(a+b)(a'+b')-|(a-b)(a'-b')|}{a'b+b'a}$. As in the proof of Claim \ref{cla:key}, the left member is $2\mu\inf\{a',b'\}\geq 2$ while the right member is at most $2$. The proof of (\ref{eq:sineratio2})-$(ii)$ is identical to that of (\ref{eq:sineratio2})-$(i)$, swapping $(a,b,x,y)$ with $(a',b',y',x')$.

\begin{flushright}
Laboratoire Paul Painlev\'e (UMR 8524) \\ 
CNRS -- Universit\'e de Lille I  \\
59 655 Villeneuve d'Ascq C\'edex, France \\
\url{Francois.Gueritaud@math.univ-lille1.fr}
\end{flushright}

\end{document}